\documentclass[10pt]{amsart}
\usepackage{amssymb,MnSymbol}
\usepackage{amsthm,amsmath}

\title[]{Algorithms and formulas for conversion between system signatures and reliability functions}

\author{Jean-Luc Marichal}
\address{Mathematics Research Unit, FSTC, University of Luxembourg, 6, rue Coudenhove-Kalergi, L-1359 Luxembourg, Luxembourg}
\email{jean-luc.marichal[at]uni.lu}

\date{April 30, 2014}

\begin{document}

\theoremstyle{plain}
\newtheorem{theorem}{Theorem}
\newtheorem{lemma}[theorem]{Lemma}
\newtheorem{proposition}[theorem]{Proposition}
\newtheorem{corollary}[theorem]{Corollary}
\newtheorem{fact}[theorem]{Fact}
\newtheorem*{main}{Main Theorem}

\theoremstyle{definition}
\newtheorem{definition}[theorem]{Definition}
\newtheorem{example}{Example}
\newtheorem{algorithm}{Algorithm}

\theoremstyle{remark}
\newtheorem*{conjecture}{onjecture}
\newtheorem{remark}{Remark}
\newtheorem{claim}{Claim}

\newcommand{\N}{\mathbb{N}}
\newcommand{\R}{\mathbb{R}}
\newcommand{\Q}{\mathbb{Q}}
\newcommand{\Vspace}{\vspace{2ex}}
\newcommand{\bfx}{\mathbf{x}}
\newcommand{\bfy}{\mathbf{y}}
\newcommand{\bfz}{\mathbf{z}}
\newcommand{\bfh}{\mathbf{h}}
\newcommand{\bfe}{\mathbf{e}}
\newcommand{\bfs}{\mathbf{s}}
\newcommand{\bfS}{\overline{\mathbf{S}}}
\newcommand{\bfd}{\mathbf{d}}
\newcommand{\os}{\mathrm{os}}
\newcommand{\dd}{\,\mathrm{d}}

\begin{abstract}
The concept of signature is a useful tool in the analysis of semicoherent systems with continuous and i.i.d.\ component lifetimes, especially for the comparison of different system designs and the computation of the system reliability. For such systems, we provide conversion formulas between the signature and the reliability function through the corresponding vector of dominations and we derive efficient algorithms for the computation of any of these concepts from the other. We also show how the signature can be easily computed from the reliability function via basic manipulations such as differentiation, coefficient extraction, and integration.
\end{abstract}



\maketitle

\section{Introduction}

Consider an $n$-component system $(C,\phi)$, where $C$ is the set $[n]=\{1,\ldots,n\}$ of its components and $\phi\colon\{0,1\}^n\to\{0,1\}$ is its structure function which expresses the state of the system in terms of the states of its components. We assume that the system is \emph{semicoherent}, which means that the structure function $\phi$ is nondecreasing in each variable and satisfies the conditions $\phi(0,\ldots,0)=0$ and $\phi(1,\ldots,1)=1$. We also assume that the components have continuous and i.i.d.\ lifetimes $T_1,\ldots,T_n$.

Samaniego \cite{Sam85} introduced the \emph{signature} of such a system as the $n$-vector $\mathbf{s}=(s_1,\ldots,s_n)$ whose $k$-th coordinate $s_k$ is the probability that the $k$-th component failure causes the system to fail. That is,
$$
s_k ~=~ \Pr(T_S=T_{k:n}),\qquad k=1,\ldots,n,
$$
where $T_S$ denotes the system lifetime and $T_{k:n}$ denotes the $k$-th smallest lifetime. From this definition one can immediately derive the identity $\sum_{k=1}^ns_k=1$.

It is very often convenient to express the signature vector $\bfs$ in terms of the {\em tail signature} of the system, a concept introduced by Boland~\cite{Bol01} and named so by Gertsbakh et al.~\cite{GerShpSpi11}. The tail signature of the system is the $(n+1)$-vector $\bfS=(\overline{S}_0,\ldots,\overline{S}_n)$ defined from $\bfs$ by
\begin{equation}\label{eq:SFromS}
\overline{S}_k ~=~\sum_{i=k+1}^ns_i{\,},\qquad k=0,\ldots,n.
\end{equation}
In particular, we have $\overline{S}_0=1$ and $\overline{S}_n=0$. Moreover, it is clear that the signature $\bfs$ can be retrieved from the tail signature $\bfS$ through the formula
\begin{equation}\label{eq:sFromS}
s_k ~=~ \overline{S}_{k-1}-\overline{S}_k{\,},\qquad k=1,\ldots,n.
\end{equation}

Recall also that the \emph{reliability function} associated with the structure function $\phi$ is the unique multilinear polynomial function $h\colon [0,1]^n\to\R$ whose restriction to $\{0,1\}^n$ is precisely the structure function $\phi$. Since the component lifetimes are independent, this function expresses the reliability of the system in terms of the component reliabilities (for general background see \cite[Chap.~2]{BarPro81} and for a more recent reference see \cite[Section 3.2]{Ram90}).

By identifying the variables of the reliability function, we obtain a real polynomial function $h(x)$ of degree at most $n$. The $n$-vector $\bfd=(d_1,\ldots,d_n)$ whose $k$-th coordinate $d_k$ is the coefficient of $x^k$ in $h(x)$ is called the \emph{vector of domination} of the system (see, e.g., \cite[Sect.~6.2]{Sam07}).

The computation of the signature of a large system by means of the usual methods may be cumbersome and tedious since it requires the evaluation of the structure function $\phi$ at every element of $\{0,1\}^n$. However, Boland et al.~\cite{BolSamVes03} observed that the $n$-vectors $\bfs$ and $\bfd$ can always be computed from each other through simple bijective linear transformations (see also \cite[Sect.~6.3]{Sam07}). Although these linear transformations were not given explicitly, they show that the signature vector $\bfs$ can be efficiently computed from the domination vector $\bfd$, or equivalently, from the polynomial function $h(x)$. Since Eqs.~(\ref{eq:SFromS}) and (\ref{eq:sFromS}) provide linear conversion formulas between vectors $\bfs$ and $\bfS$, we observe that any of the vectors $\bfs$, $\bfS$, and $\bfd$ can be computed from any other by means of a bijective linear transformation (see Figure~\ref{fig:blt}).

\setlength{\unitlength}{5ex}
\begin{figure}[htbp]\centering
\begin{picture}(4.5,2.5)
\put(0.25,0.25){\makebox(0,0)[t]{$\bfs$}}\put(4.25,0.25){\makebox(0,0){$\bfS$}}\put(2.25,2.25){\makebox(0,0){$\bfd$ or $h(x)$}}
\put(0.50,0.75){\vector(1,1){1}}\put(1.75,1.50){\vector(-1,-1){1}}
\put(3.75,0.50){\vector(-1,1){1}}\put(3,1.75){\vector(1,-1){1}}
\put(1.25,0.30){\vector(1,0){2}}\put(3.25,-0.03){\vector(-1,0){2}}
\end{picture}
\caption{Bijective linear transformations} \label{fig:blt}
\end{figure}
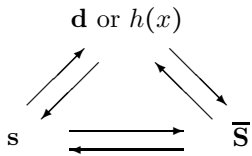

After recalling some basic formulas in Section 2 of this paper, in Section 3 we yield these linear transformations explicitly and present them as linear conversion formulas. From these conversion formulas we derive algorithms for the computation of any of these vectors from any other. These algorithms prove to be very efficient since they require at most $\frac{1}{2}n(n+1)$ additions and multiplications.

We also show how the computation of the vectors $\bfs$ and $\bfS$ can be easily performed from basic manipulations of function $h(x)$ such as differentiation, reflection, coefficient extraction, and integration. For instance, we establish the polynomial identity (see Eq.~(\ref{eq:sfrrd6543}))
\begin{equation}\label{eq:exsg3}
\sum_{k=1}^n{n\choose k}{\,}s_k{\,}x^k ~=~ \int_0^x(R^{n-1}h')(t+1)\, dt{\,},
\end{equation}
where $h'(x)$ is the derivative of $h(x)$ and $(R^{n-1}h')(x)$ is the polynomial function obtained from $h'(x)$ by switching the coefficients of $x^k$ and $x^{n-1-k}$ for $k=0,\ldots,n-1$. Applying this result to the classical $5$-component bridge system (see Example~\ref{ex:21j4hg} below), we can easily see that Eq.~(\ref{eq:exsg3}) reduces to
$$
5s_1{\,}x + 10s_2{\,}x^2 + 10s_3{\,}x^3 + 5s_4{\,}x^4 + s_5{\,}x^5 ~=~ 2x^2+6x^3+x^4{\,}.
$$
By equating the corresponding coefficients we immediately obtain the signature vector $\bfs=(0,\frac{1}{5},\frac{3}{5},\frac{1}{5},0)$.

In Section 4 we examine the general non-i.i.d.\ setting where the component lifetimes $T_1,\ldots,T_n$ may be dependent. We show how a certain modification of the structure function enables us to formally extend almost all the conversion formulas and algorithms obtained in Sections 2 and 3 to the general dependent setting.
Finally, we end our paper in Section 5 by some concluding remarks.

\section{Preliminaries}

Boland \cite{Bol01} showed that every coordinate $s_k$ of the signature vector can be explicitly written in the form
\begin{equation}\label{eq:asad678}
s_k ~=~ \sum_{\textstyle{A\subseteq C\atop |A|=n-k+1}}\frac{1}{{n\choose |A|}}\,\phi(A)-\sum_{\textstyle{A\subseteq C\atop |A|=n-k}}\frac{1}{{n\choose |A|}}\,\phi(A)\, .
\end{equation}
Here and throughout we identify Boolean $n$-vectors $\bfx\in\{0,1\}^n$ and subsets $A\subseteq [n]$ in the usual way, that is, by setting $x_i=1$ if and only if $i\in A$. Thus we use the same symbol to denote both a function $f\colon\{0,1\}^n\to\R$ and the corresponding set function $f\colon 2^{[n]}\to\R$ interchangeably. For instance, we write $\phi(0,\ldots,0)=\phi(\varnothing)$ and $\phi(1,\ldots,1)=\phi(C)$.

As mentioned in the introduction, the reliability function associated with the structure function $\phi$ is the multilinear function $h\colon [0,1]^n\to\R$ defined by
\begin{equation}\label{eq:345hg}
h(\bfx) ~=~ h(x_1,\ldots,x_n) ~=~ \sum_{A\subseteq C}\phi(A)\, \prod_{i\in A}x_i\,\prod_{i\in C\setminus A}(1-x_i).
\end{equation}
It is easy to see that this function can always be put in the unique standard multilinear form
\begin{equation}\label{eq:345hg2}
h(\bfx) ~=~ \sum_{A\subseteq C}d(A)\,\prod_{i\in A}x_i{\,},
\end{equation}
where, for every $A\subseteq C$, the coefficient $d(A)$ is an integer.

By identifying the variables $x_1,\ldots,x_n$ in function $h(\bfx)$, we define its diagonal section $h(x,\ldots,x)$, which we have simply denoted by $h(x)$. From Eqs.~(\ref{eq:345hg}) and
(\ref{eq:345hg2}) we immediately obtain
$$
h(x) ~=~ \sum_{A\subseteq C}\phi(A)\, x^{|A|} (1-x)^{n-|A|} ~=~ \sum_{A\subseteq C}d(A)\, x^{|A|}\, ,
$$
or equivalently,
\begin{equation}\label{eq:swrdf675sh}
h(x) ~=~ \sum_{k=0}^n\phi_k\, x^{k}(1-x)^{n-k} ~=~ \sum_{k=0}^n d_k\, x^{k}\, ,
\end{equation}
where
\begin{equation}\label{eq:sdf675sh}
\phi_k ~=~ \sum_{\textstyle{A\subseteq C\atop |A|=k}}\phi(A)\qquad\mbox{and}\qquad d_k ~=~ \sum_{\textstyle{A\subseteq C\atop |A|=k}}d(A){\,},\qquad k=0,\ldots,n.
\end{equation}

Clearly, we have $\phi_0=\phi(\varnothing)=0$ and $d_0=d(\varnothing)=h(0)=0$. As already mentioned, the $n$-vector $\bfd=(d_1,\ldots,d_n)$ is called the \emph{vector of dominations} of the system.

\begin{example}\label{ex:21j4hg}
Consider the bridge structure as indicated in Figure~\ref{fig:bs}. The corresponding structure function is given by
$$
\phi(x_1,\ldots,x_5) ~=~ x_1\, x_4\amalg x_2\, x_5\amalg x_1\, x_3\, x_5\amalg x_2\, x_3\, x_4{\,},
$$
where $\amalg$ is the (associative) coproduct operation defined by $x\amalg y = 1-(1-x)(1-y)$. The corresponding reliability function, given in Eq.~(\ref{eq:345hg}), can be computed by expanding the coproducts in $\phi$ and then simplifying the resulting algebraic expression using $x_i^2=x_i$. We have
\begin{eqnarray*}
h(x_1,\ldots,x_5) &=& x_1 x_4 + x_2 x_5 + x_1 x_3 x_5 + x_2 x_3 x_4 \\
&& \null - x_1 x_2 x_3 x_4 - x_1 x_2 x_3 x_5 - x_1 x_2 x_4 x_5  - x_1 x_3 x_4 x_5 - x_2 x_3 x_4 x_5\\
&& \null + 2\, x_1 x_2 x_3 x_4 x_5\, .
\end{eqnarray*}
We then obtain its diagonal section $h(x) = 2x^2+2x^3-5x^4+2x^5$ and finally the domination vector $\bfd=(0,2,2,-5,2)$.
\end{example}

\setlength{\unitlength}{4ex}
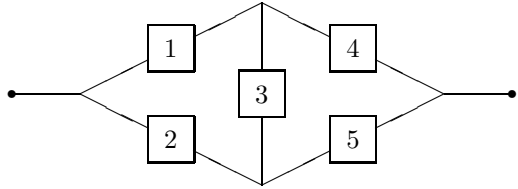
\begin{figure}[htbp]\centering
\begin{picture}(11,4)
\put(3,0.5){\framebox(1,1){$2$}} \put(3,2.5){\framebox(1,1){$1$}} \put(5,1.5){\framebox(1,1){$3$}} \put(7,0.5){\framebox(1,1){$5$}}
\put(7,2.5){\framebox(1,1){$4$}}%
\put(0,2){\line(1,0){1.5}}\put(1.5,2){\line(2,-1){1.5}}\put(5.5,0){\line(-2,1){1.5}}\put(1.5,2){\line(2,1){1.5}}\put(5.5,4){\line(-2,-1){1.5}}%
\put(0,2){\circle*{0.15}}%
\put(9.5,2){\line(1,0){1.5}}\put(5.5,0){\line(2,1){1.5}}\put(9.5,2){\line(-2,-1){1.5}}\put(5.5,4){\line(2,-1){1.5}}\put(9.5,2){\line(-2,1){1.5}}%
\put(11,2){\circle*{0.15}}%
\put(5.5,0){\line(0,1){1.5}}\put(5.5,4){\line(0,-1){1.5}}
\end{picture}
\caption{Bridge structure} \label{fig:bs}
\end{figure}

Example~\ref{ex:21j4hg} illustrates the important fact that the reliability function $h(\bfx)$ of any system can be easily obtained from the minimal path sets simply by first expressing the structure function as a coproduct over the minimal path sets and then expanding the coproduct and simplifying the resulting algebraic expression (using $x_i^2=x_i$) until it becomes multilinear. The diagonal section $h(x)$ of the reliability function is then obtained by identifying all the variables.

This observation is crucial since, when combined with an efficient algorithm for converting the polynomial function $h(x)$ into the signature vector, it provides an efficient way to compute the signature of any system from its minimal path sets.

\section{Conversion formulas}

Recall that Eq.~(\ref{eq:345hg2}) gives the standard multilinear form of the reliability function $h(\bfx)$. As mentioned for instance in \cite[p.~31]{Ram90}, the link between the coefficients $d(A)$ and the values $\phi(A)$ is given through the following linear conversion formulas (obtained from the M\"obius inversion theorem)
\begin{equation}\label{eq:Mobius}
\phi(A) ~=~ \sum_{B\subseteq A}d(B)\qquad\mbox{and}\qquad d(A) ~=~ \sum_{B\subseteq A}(-1)^{|A|-|B|}\, \phi(B)\, .
\end{equation}

The following proposition yields the linear conversion formulas between the $n$-vectors $\bfd=(d_1,\ldots,d_n)$ and $(\phi_1,\ldots,\phi_n)$. Note that an alternative form of Eq.~(\ref{eq:Mobiusc1}) was previously found by Samaniego \cite[Sect.~6.3]{Sam07}.

\begin{proposition}
We have
\begin{equation}\label{eq:Mobiusc}
\phi_k ~=~ \sum_{j=0}^k{n-j\choose k-j}{\,}d_j{\,},\qquad k=1,\ldots,n,
\end{equation}
and
\begin{equation}\label{eq:Mobiusc1}
d_k ~=~ \sum_{j=0}^k(-1)^{k-j}{\,}{n-j\choose k-j}{\,}\phi_j{\,},\qquad k=1,\ldots,n.
\end{equation}
\end{proposition}

\begin{proof}
By Eqs.~(\ref{eq:sdf675sh}) and (\ref{eq:Mobius}) we have
$$
\phi_k ~=~ \sum_{\textstyle{A\subseteq C\atop |A|=k}}\phi(A) ~=~ \sum_{\textstyle{A\subseteq C\atop |A|=k}}\sum_{B\subseteq A}d(B).
$$
Permuting the sums and then setting $j=|B|$, we obtain
$$
\phi_k ~=~ \sum_{\textstyle{B\subseteq C\atop |B|\leqslant k}}d(B)\sum_{\textstyle{A\supseteq B\atop |A|=k}} 1 ~=~ \sum_{\textstyle{B\subseteq C\atop |B|\leqslant k}}{n-|B|\choose k-|B|}{\,}d(B) ~=~ \sum_{j=0}^k{n-j\choose k-j}{\,}\sum_{\textstyle{B\subseteq C\atop |B|=j}}d(B),
$$
which proves Eq.~(\ref{eq:Mobiusc}). Formula~(\ref{eq:Mobiusc1}) can be established similarly.
\end{proof}

We are now ready to establish conversion formulas and algorithms as announced in the introduction.

\subsection{Conversions between $\bfs$ and $\bfS$}

We already know that the linear conversion formulas between the vectors $\bfs$ and $\bfS$ are given by Eqs.~(\ref{eq:SFromS}) and (\ref{eq:sFromS}). This conversion can also be explicitly expressed by means of a polynomial identity. Let $\sum_{k=1}^ns_k\, x^k$ and $\sum_{k=0}^n\overline{S}_k\, x^k$ be the {\em generating functions} of vectors $\bfs$ and $\bfS$, respectively. Then we have the polynomial identity
\begin{equation}\label{eq:sfd76dd}
\sum_{k=1}^ns_k\, x^k ~=~ 1+(x-1)\,\sum_{k=0}^n\overline{S}_k\, x^k.
\end{equation}
Indeed, using Eq.~(\ref{eq:sFromS}) and summation by parts, we obtain
$$
\sum_{k=1}^ns_k\, x^k ~=~ \sum_{k=1}^n\big(\overline{S}_{k-1}-\overline{S}_k\big)\, x^k ~=~ x+\sum_{k=1}^n\overline{S}_k\, \big(x^{k+1}-x^k\big)\, ,
$$
which proves Eq.~(\ref{eq:sfd76dd}).

For instance, for the bridge system described in Example~\ref{ex:21j4hg}, the generating functions of vectors $\bfs$ and $\bfS$ are given by $\frac{1}{5}{\,}x^2 + \frac{3}{5}{\,}x^3 + \frac{1}{5}{\,}x^4$ and $1 + x + \frac{4}{5}{\,}x^2 + \frac{1}{5}{\,}x^3$, respectively. We can easily verify that Eq.~(\ref{eq:sfd76dd}) holds for these functions.

\subsection{Conversions between $\bfS$ and $\bfd$}

Combining Eq.~(\ref{eq:SFromS}) with Eqs.~(\ref{eq:asad678}) and (\ref{eq:sdf675sh}), we observe that
\begin{equation}\label{eq:TailSignature2}
\overline{S}_k ~=~ \frac{1}{{n\choose k}}\,\sum_{\textstyle{A\subseteq C\atop |A|=n-k}}\,\phi(A) ~=~ \frac{1}{{n\choose k}}\,\phi_{n-k}{\,},\qquad k=0,\ldots,n.
\end{equation}

Recall that a \emph{path set} of the system is a component subset $A$ such that $\phi(A)=1$. It follows from Eq.~(\ref{eq:TailSignature2}) that $\phi_k$ is precisely the number of path sets of size $k$ and that $\overline{S}_{n-k}$ is the proportion of component subsets of size $k$ which are path sets. We also observe that the leading coefficient $d_n$ of $h(\bfx)$, also known as the \emph{signed domination} \cite{BarIye88} of $h(\bfx)$, is zero if and only if there are as many path sets of odd sizes as path sets of even sizes. This observation immediately follows from the identity $d_n=\sum_{j=0}^n (-1)^{n-j}{\,}\phi_j$, obtained by setting $k=n$ in Eq.~(\ref{eq:Mobiusc1}).

Combining Eqs.~(\ref{eq:Mobiusc}) and (\ref{eq:Mobiusc1}) with Eq.~(\ref{eq:TailSignature2}), we immediately obtain the following conversion formulas between the vectors $\bfS$ and $\bfd$.

\begin{proposition}
We have
\begin{eqnarray}
\overline{S}_k &=& \sum_{j=0}^{n-k}\frac{{n-j\choose k}}{{n\choose k}}\, d_j ~=~ \sum_{j=0}^{n-k}\frac{{n-k\choose j}}{{n\choose j}}\, d_j{\,},\qquad k=0,\ldots,n,\label{eq:fuji1}\\
d_k &=& {n\choose k}\sum_{j=0}^k(-1)^{k-j}{\,}{k\choose j}{\,}\overline{S}_{n-j}{\,},\qquad k=0,\ldots,n.\label{eq:fuji2}
\end{eqnarray}
\end{proposition}

Equation~(\ref{eq:fuji2}) can be rewritten in a simpler form by using the classical difference operator $\Delta_i$ which maps a sequence $z_i$ to the sequence $\Delta_iz_i=z_{i+1}-z_i$. Defining the $k$-th difference $\Delta_i^k z_i$ of a sequence $z_i$ recursively as $\Delta_i^0z_i=z_i$ and $\Delta_i^kz_i=\Delta_i^{\mathstrut}\Delta_i^{k-1}z_i$, we can show by induction on $k$ that
\begin{equation}\label{eq:ddetp}
\Delta_i^kz_i ~=~ \sum_{j=0}^k(-1)^{k-j}{\,}{k\choose j}{\,}z_{i+j}{\,}.
\end{equation}
Comparing Eq.~(\ref{eq:fuji2}) with Eq.~(\ref{eq:ddetp}) immediately shows that Eq.~(\ref{eq:fuji2}) can be rewritten as
\begin{equation}\label{eq:fuji2p}
d_k ~=~ {n\choose k}\,\big(\Delta_i^k{\,}\overline{S}_{n-i}\big)\big|_{i=0}{\,},\qquad k=1,\ldots,n,
\end{equation}
and the vector $\bfd$ can then be computed efficiently from a classical difference table (see Table~\ref{tab:112a}).

\begin{table}[htbp]
$$
\begin{array}{c|ccc}
\overline{S}_n & & & \\
 & {n\choose 1}(\Delta_i\overline{S}_{n-i})|_{i=0} & & \\
\overline{S}_{n-1} & & {n\choose 2}(\Delta^2_i\overline{S}_{n-i})|_{i=0} & \\
 & {n\choose 1}(\Delta_i\overline{S}_{n-i})|_{i=1} & & {n\choose 3}(\Delta_i^3\overline{S}_{n-i})|_{i=0} \\
\overline{S}_{n-2} & & {n\choose 2}(\Delta^2_i\overline{S}_{n-i})|_{i=1} & \vdots \\
 & {n\choose 1}(\Delta_i\overline{S}_{n-i})|_{i=2} & \vdots & \\
\overline{S}_{n-3} & \vdots & & \\
\vdots & & &
\end{array}
$$
\caption{Computation of $\bfd$ from $\bfS$}
\label{tab:112a}
\end{table}

Setting $D_{j,k}={n\choose k}(\Delta^k_i{\,}\overline{S}_{n-i})|_{i=j}$, from Eq.~(\ref{eq:fuji2p}) we can easily derive the following algorithm for the computation of $\bfd$. This algorithm requires only $\frac{1}{2}n(n+1)$ additions and multiplications.

\begin{algorithm}\label{algo:s4ex}
The following algorithm inputs vector $\bfS$ and outputs vector $\bfd$. It uses the variables $D_{j,k}$ for $k=0,\ldots,n$ and $j=0,\ldots,n-k$.

\smallskip

\begin{quote}
\begin{enumerate}
\item[\textbf{Step 1.}] For $j=0,\ldots,n$, set $D_{j,0}:=\overline{S}_{n-j}$.

\item[\textbf{Step 2.}] For $k=1,\ldots,n$

\mbox{}\hspace{3ex}For $j=0,\ldots,n-k$

\mbox{}\hspace{6ex}$D_{j,k} ~:=~ \frac{n-k+1}{k}\,(D_{j+1,k-1}-D_{j,k-1})$

\item[\textbf{Step 3.}] For $k=0,\ldots,n$, set $d_k:=D_{0,k}$.
\end{enumerate}
\end{quote}

\smallskip
\end{algorithm}

\begin{example}\label{eq:sda68}
Consider the bridge system described in Example~\ref{ex:21j4hg}. The corresponding tail signature vector is given by $\bfS=(1,1,\frac{4}{5},\frac{1}{5},0,0)$. Forming the difference table (see Table~\ref{tab:112ab}) and reading its first row, we obtain the vector $\bfd=(0,2,2,-5,2)$ and therefore the function $h(x)=2x^2+2x^3-5x^4+2x^5$.
\end{example}

\begin{table}[htbp]
$$
\begin{array}{c|ccccc}
0 & & & & &\\
 & 0 & & & &\\
0 & & 2 & & &\\
 & 1 & & 2 & &\\
1/5 & & 4 & & -5 &\\
 & 3 & & -8 & & 2\\
4/5 & & -4 & & 5 & \\
 & 1 & & 2 & & \\
1 & & -2 & & & \\
& 0 & & & &\\
1 & & & & &
\end{array}
$$
\caption{Computation of $\bfd$ from $\bfS$ (Example~\ref{eq:sda68})}
\label{tab:112ab}
\end{table}

The converse transformation (\ref{eq:fuji1}) can then be computed efficiently by the following algorithm, in which we compute the quantities
$$
S_{j,k} ~=~ \sum_{i=0}^k \frac{{k\choose i}{i+j\choose i}}{{n-j\choose i}}{\,}d_{i+j}{\,}.
$$

\begin{algorithm}\label{algo:s4ex}
The following algorithm inputs vector $\bfd$ and outputs vector $\bfS$. It uses the variables $S_{j,k}$ for $k=0,\ldots,n$ and $j=0,\ldots,n-k$.

\smallskip

\begin{quote}
\begin{enumerate}
\item[\textbf{Step 1.}] For $j=0,\ldots,n$, set $S_{j,0}:=d_j$.

\item[\textbf{Step 2.}] For $k=1,\ldots,n$

\mbox{}\hspace{3ex}For $j=0,\ldots,n-k$

\mbox{}\hspace{6ex}$S_{j,k} ~:=~ \frac{j+1}{n-j}{\,}S_{j+1,k-1}+S_{j,k-1}$

\item[\textbf{Step 3.}] For $k=0,\ldots,n$, set $\overline{S}_{n-k}:=S_{0,k}$.
\end{enumerate}
\end{quote}

\smallskip
\end{algorithm}

\subsection{Conversions between $\bfs$ and $\bfd$}

The following proposition yields the conversion formulas between the vectors $\bfs$ and $\bfd$. Note that a non-explicit version of Eq.~(\ref{eq:fuji3}) was previously found in Boland et al.~\cite{BolSamVes03} (see also Theorem~6.1 in \cite{Sam07}).

\begin{proposition}
We have
\begin{eqnarray}
s_k &=& \sum_{j=0}^{n-k}\frac{{n-j\choose k}}{{n\choose k}}\,\frac{j+1}{n-j}{\,}d_{j+1} ~=~ \sum_{j=1}^{n-k+1}\frac{{n-k\choose j-1}}{{n\choose j}}{\,}d_{j}{\,},\qquad k=1,\ldots,n,\label{eq:fuji3}\\
d_k &=& {n\choose k}\sum_{j=0}^{k-1}(-1)^{k-1-j}{\,}{k-1\choose j}{\,}s_{n-j}{\,},\qquad k=1,\ldots,n.\label{eq:fuji5}\\
d_k &=& {n\choose k}\,\big(\Delta_i^{k-1}{\,}s_{n-i}\big)\big|_{i=0}^{\mathstrut}{\,},\qquad k=1,\ldots,n,\label{eq:fuji4}
\end{eqnarray}
\end{proposition}

\begin{proof}
Combining Eq.~(\ref{eq:fuji1}) with Eq.~(\ref{eq:sFromS}), we obtain
\begin{eqnarray*}
s_k ~=~ \overline{S}_{k-1}-\overline{S}_k &=& \sum_{j=1}^{n-k+1}\frac{{n-k+1\choose j}}{{n\choose j}}\, d_j-\sum_{j=1}^{n-k}\frac{{n-k\choose j}}{{n\choose j}}\, d_j\\
&=& \sum_{j=1}^{n-k}\frac{{n-k\choose j-1}}{{n\choose j}}\, d_j + \frac{1}{{n\choose n-k+1}}{\,}d_{n-k+1}{\,},
\end{eqnarray*}
which proves Eq.~(\ref{eq:fuji3}). By Eq.~(\ref{eq:sFromS}) we have $\Delta_i\overline{S}_{n-i}=s_{n-i}$ for $i=0,\ldots,n-1$. Equation~(\ref{eq:fuji4}) then follows from Eq.~(\ref{eq:fuji2p}). Equation~(\ref{eq:fuji5}) then follows immediately from Eq.~(\ref{eq:fuji4}).
\end{proof}

Equation~(\ref{eq:fuji4}) shows that $\bfd$ can be efficiently computed directly from $\bfs$ by means of a difference table (see Table~\ref{tab:11g2a}).

\begin{table}[htbp]
$$
\begin{array}{c|ccc}
{n\choose 1}s_n & & & \\
 & {n\choose 2}(\Delta_is_{n-i})|_{i=0} & & \\
{n\choose 1}s_{n-1} & & {n\choose 3}(\Delta^2_is_{n-i})|_{i=0} & \\
 & {n\choose 2}(\Delta_is_{n-i})|_{i=1} & & {n\choose 4}(\Delta_i^3s_{n-i})|_{i=0} \\
{n\choose 1}s_{n-2} & & {n\choose 3}(\Delta^2_is_{n-i})|_{i=1} & \vdots \\
 & {n\choose 2}(\Delta_is_{n-i})|_{i=2} & \vdots & \\
{n\choose 1}s_{n-3} & \vdots & & \\
\vdots & & &
\end{array}
$$
\caption{Computation of $\bfd$ from $\bfs$}
\label{tab:11g2a}
\end{table}

Setting $d_{j,k}={n\choose k}(\Delta^{k-1}_is_{n-i})|_{i=j-1}$, we can also derive the following algorithm for the computation of vector $\bfd$. This algorithm requires only $\frac{1}{2}n(n-1)$ additions and multiplications.

\begin{algorithm}\label{algo:s1}
The following algorithm inputs vector $\bfs$ and outputs vector $\bfd$. It uses the variables $d_{j,k}$ for $k=1,\ldots,n$ and $j=1,\ldots,n-k+1$.

\smallskip

\begin{quote}
\begin{enumerate}
\item[\textbf{Step 1.}] For $j=1,\ldots,n$, set $d_{j,1}:=n{\,}s_{n-j+1}$.

\item[\textbf{Step 2.}] For $k=2,\ldots,n$

\mbox{}\hspace{3ex}For $j=1,\ldots,n-k+1$

\mbox{}\hspace{6ex}$d_{j,k} ~:=~ \frac{n-k+1}{k}{\,}(d_{j+1,k-1}-d_{j,k-1})$

\item[\textbf{Step 3.}] For $k=1,\ldots,n$, set $d_k := d_{1,k}$.
\end{enumerate}
\end{quote}

\smallskip
\end{algorithm}

\begin{example}\label{eq:sda6f8}
Consider again the bridge system described in Example~\ref{ex:21j4hg}. The corresponding signature vector is given by $\bfs=(0,\frac{1}{5},\frac{3}{5},\frac{1}{5},0)$. Forming the difference table (see Table~\ref{tab:11f2ab}) and reading its first row, we obtain the vector $\bfd=(0,2,2,-5,2)$ and hence the function $h(x)=2x^2+2x^3-5x^4+2x^5$.
\end{example}

\begin{table}[htbp]
$$
\begin{array}{c|cccc}
0 & & & &\\
 & 2 & & &\\
1 & & 2 & &\\
 & 4 & & -5 &\\
3 & & -8 & & 2\\
 & -4 & & 5 &\\
1 & & 2 & &\\
 & -2 & & &\\
0 & & & &
\end{array}
$$
\caption{Computation of $\bfd$ from $\bfs$ (Example~\ref{eq:sda6f8})}
\label{tab:11f2ab}
\end{table}

The converse transformation (\ref{eq:fuji3}) can then be computed efficiently by the following algorithm, in which we compute the quantities
$$
s_{j,k} ~=~ \frac{1}{n}\sum_{i=1}^{k}\frac{{k-1\choose i-1}{i+j-1\choose i-1}}{{n-j\choose i-1}}{\,}d_{i+j-1}{\,}.
$$

\begin{algorithm}\label{algo:s4ex6}
The following algorithm inputs vector $\bfd$ and outputs vector $\bfs$. It uses the variables $s_{j,k}$ for $k=1,\ldots,n$ and $j=1,\ldots,n-k+1$.

\smallskip

\begin{quote}
\begin{enumerate}
\item[\textbf{Step 1.}] For $j=1,\ldots,n$, set $s_{j,1}:=\frac{1}{n}{\,}d_j$.

\item[\textbf{Step 2.}] For $k=2,\ldots,n$

\mbox{}\hspace{3ex}For $j=1,\ldots,n-k+1$

\mbox{}\hspace{6ex}$s_{j,k} ~:=~ \frac{j+1}{n-j}{\,}s_{j+1,k-1}+s_{j,k-1}$

\item[\textbf{Step 3.}] For $k=1,\ldots,n$, set $s_{n-k+1}:=s_{1,k}$.
\end{enumerate}
\end{quote}

\smallskip
\end{algorithm}

\subsection{Conversions between $\bfS$ or $\bfs$ and $h(x)$}

The conversion formulas between vectors $\bfs$ and $\bfd$ show that the diagonal section $h(x)$ of the reliability function encodes exactly the signature (or equivalently, the tail signature), no more, no less. Even though the latter can be computed from vector $\bfd$ using Eqs.~(\ref{eq:fuji1}) and (\ref{eq:fuji3}), we will now see how we can compute it by direct and simple algebraic manipulations of function $h(x)$.

Let $f$ be a univariate polynomial of degree $\leqslant n$,
$$
f(x) ~=~ a_n\, x^n+a_{n-1}\, x^{n-1} + \cdots + a_1\, x+ a_0\, .
$$
The $n$-\emph{reflected} of $f$ is the polynomial $R^n f$ obtained from $f$ by switching the coefficients of $x^k$ and $x^{n-k}$ for $k=0,\ldots,n$; that is,
$$
(R^n f)(x) ~=~ a_0\, x^n+a_1\, x^{n-1}+\cdots + a_{n-1}\, x + a_n\, ,
$$
or equivalently, $(R^n f)(x)=x^n\, f(1/x)$.

Combining Eq.~(\ref{eq:swrdf675sh}) with Eq.~(\ref{eq:TailSignature2}), we obtain (see also \cite{BolSamVes03})
\begin{equation}\label{eq:aersd76fds}
h(x) ~=~ \sum_{k=0}^n \overline{S}_{n-k}{\,}{n\choose k}{\,}x^k(1-x)^{n-k}.
\end{equation}
From this equation it follows, as it was already observed in \cite{MarMatb}, that
\begin{equation}\label{eq:asd76fds}
(R^nh)(x+1) ~=~ \sum_{k=0}^n{n\choose k}\overline{S}_k{\,}x^k.
\end{equation}
Thus, ${n\choose k}\overline{S}_k$ can be obtained simply by reading the coefficient of $x^k$ in the polynomial function $(R^nh)(x+1)$. Denoting by $[x^k]f(x)$ the coefficient of $x^k$ in a polynomial function $f(x)$, Eq.~(\ref{eq:asd76fds}) can be rewritten as
\begin{equation}\label{eq:lj3kl1}
{n\choose k}\overline{S}_k ~=~ [x^k](R^nh)(x+1),\qquad k=0,\ldots,n.
\end{equation}
From Eq.~(\ref{eq:lj3kl1}) we immediately derive the following algorithm (see also \cite{MarMatb}).

\begin{algorithm}\label{algo:s55}
The following algorithm inputs $n$ and $h(x)$ and outputs $\bfS$.

\smallskip

\begin{quote}
\begin{enumerate}
\item[\textbf{Step 1.}] For $k=0,\ldots,n$, let $a_k$ be the coefficient of $x^k$ in the $n$-degree polynomial $(R^nh)(x+1) = (x+1)^n\, h\big(\frac{1}{x+1}\big)$.

\item[\textbf{Step 2.}] We have $\overline{S}_k = a_k/{n\choose k}$ for $k=0,\ldots,n$.
\end{enumerate}
\end{quote}

\smallskip
\end{algorithm}

The following proposition yields the analog of Eqs.~(\ref{eq:asd76fds}) and (\ref{eq:lj3kl1}) for the signature. Here and throughout we denote by $h'(x)$ the derivative of $h(x)$.

\begin{proposition}\label{prop:bth49}
We have
\begin{eqnarray}
&& k{\,}{n\choose k}{\,}s_k ~=~ [x^{k-1}](R^{n-1}h')(x+1),\qquad k=1,\ldots,n,\label{eq:s87sf6sd}\\
&& \sum_{k=1}^n {n\choose k}\, s_k\, k\, x^{k-1} ~=~ (R^{n-1}h')(x+1)\, ,\label{eq:sfd6543}\\
&& \sum_{k=1}^n{n\choose k}{\,}s_k{\,}x^k ~=~ \int_0^x(R^{n-1}h')(t+1)\, dt\, .\label{eq:sfrrd6543}
\end{eqnarray}
\end{proposition}

\begin{proof}
By Eq.~(\ref{eq:swrdf675sh}) we have $h'(x) = \sum_{j=0}^{n-1}(j+1){\,}d_{j+1}{\,}x^j$ and therefore
\begin{eqnarray*}
(R^{n-1}h')(x+1) &=& \sum_{j=0}^{n-1}(j+1){\,}d_{j+1}(x+1)^{n-1-j}\\
&=& \sum_{j=0}^{n-1}(j+1){\,}d_{j+1}\sum_{k=1}^{n-j}{n-1-j\choose k-1}{\,}x^{k-1}\\
&=& \sum_{k=1}^nx^{k-1}\sum_{j=0}^{n-k}{n-1-j\choose k-1}(j+1){\,}d_{j+1}.
\end{eqnarray*}
Thus, the inner sum in the latter expression is the coefficient of $x^{k-1}$ in the polynomial function $(R^{n-1}h')(x+1)$. Dividing this sum by $k{n\choose k}$ and then using Eq.~(\ref{eq:fuji3}), we obtain $s_k$. This proves Eqs.~(\ref{eq:s87sf6sd}) and (\ref{eq:sfd6543}). Equation (\ref{eq:sfrrd6543}) is then obtained by integrating both sides of Eq.~(\ref{eq:sfd6543}) on the interval $[0,x]$.
\end{proof}

From Eq.~(\ref{eq:s87sf6sd}) we immediately derive the following algorithm.

\begin{algorithm}\label{algo:s}
The following algorithm inputs $n$ and $h(x)$ and outputs $\mathbf{s}$.

\smallskip

\begin{quote}
\begin{enumerate}
\item[\textbf{Step 1.}] For $k=1,\ldots,n$, let $a_{k-1}$ be the coefficient of $x^{k-1}$ in the $(n-1)$-degree polynomial $(R^{n-1}h')(x+1) = (x+1)^{n-1}\, h'\big(\frac{1}{x+1}\big)$.

\item[\textbf{Step 2.}] We have $s_k = a_{k-1}/(k\,{n\choose k})$ for $k=1,\ldots,n$.
\end{enumerate}
\end{quote}

\smallskip
\end{algorithm}

Even though such an algorithm can be easily executed by hand for small $n$, a computer algebra system can be of great assistance for large $n$.

\begin{example}\label{ex:21j4hg3z}
Consider again the bridge system described in Example~\ref{ex:21j4hg}. We have
$$
h'(x) ~=~ 4x+6x^2-20x^3+10x^4\quad\mbox{and}\quad (R^4h')(x) ~=~ 10-20x+6x^2+4x^3.
$$
It follows that $(R^4h')(x+1)=4 x + 18 x^2 + 4 x^3$ and hence $\mathbf{s}=\big(0,\frac{1}{5},\frac{3}{5},\frac{1}{5},0\big)$ by Algorithm~\ref{algo:s}. Indeed, we have for instance $s_3=a_2/(3{5\choose 3}) =\frac{3}{5}$.
\end{example}

The following proposition, established in \cite{MarMatb}, provides a necessary and sufficient condition on the system signature for the reliability function to be of full degree (i.e., the corresponding signed domination $d_n$ is nonzero). Here we provide a shorter proof based on Eq.~(\ref{eq:sfd6543}).

\begin{proposition}[{\cite{MarMatb}}]\label{prop:sad6f}
Let $(C,\phi)$ be an $n$-component semicoherent system with continuous and i.i.d.\ component lifetimes. Then the reliability function $h(\bfx)$ (or equivalently, its
diagonal section $h(x)$) is a polynomial of degree $n$ if and only if
$$
\sum_{\mbox{$k$ odd}}{n-1\choose k-1}\, s_k ~\neq ~\sum_{\mbox{$k$ even}}{n-1\choose k-1}\, s_k\, .
$$
\end{proposition}

\begin{proof}
The function $h(x)$ is of degree $n$ if and only if $h'(x)$ is of degree $n-1$ and this condition holds if and only if $d_n=\frac{1}{n}(R^{n-1}h')(0)\neq 0$. By Eq.~(\ref{eq:sfd6543}) this means that
$$
\sum_{k=1}^n {n\choose k}\, s_k\, k\, (-1)^{k-1} ~=~ n\,\sum_{k=1}^n {n-1\choose k-1}\, s_k\, (-1)^{k-1}
$$
is not zero.
\end{proof}

The vectors $\bfs$ and $\bfS$ can also be computed via their generating functions. The following proposition yields integral formulas for these functions.

\begin{proposition}
We have
\begin{eqnarray}
\sum_{k=0}^n \overline{S}_k\, x^k &=& \int_0^1 (n+1)\, R^n_t\, \big((R^{n}h)((t-1)\, x+1)\big)\, dt\, ,\label{eq:sfd76ddd}\\
\sum_{k=1}^n s_k\, x^k &=& \int_0^1 x\, R^{n-1}_t\, \big((R^{n-1}h')((t-1)\, x+1)\big)\, dt\, ,\label{prop:sdf67g}
\end{eqnarray}
where $R_t^n$ is the $n$-reflection with respect to variable $t$.
\end{proposition}

\begin{proof}
By Eq.~(\ref{eq:asd76fds}), we have
$$
(R^n h)((t-1){\,}x+1) ~=~ \sum_{k=0}^n {n\choose k}\, \overline{S}_k\, (t-1)^{k}x^k
$$
and hence
$$
R^n_t\, \big((R^{n}h)((t-1)\, x+1)\big) ~=~ \sum_{k=0}^n {n\choose k}\, \overline{S}_k\, t^{n-k}\,(1-t)^{k}x^k.
$$
Integrating this expression from $t=0$ to $t=1$ and using the well-known identity
\begin{equation}\label{eq:Betakn}
\int_0^1t^{n-k}\,(1-t)^{k}\, dt ~=~ \frac{1}{(n+1){n\choose k}}{\,},
\end{equation}
we finally obtain Eq.~(\ref{eq:sfd76ddd}). Formula (\ref{prop:sdf67g}) can be proved similarly by using Eq.~(\ref{eq:sfd6543}).
\end{proof}

From Eq.~(\ref{prop:sdf67g}) we immediately derive the following algorithm for the computation of the generating function of the signature. The algorithm corresponding to Eq.~(\ref{eq:sfd76ddd}) can be derived similarly.

\begin{algorithm}\label{algo:gfsss}
The following algorithm inputs $n$ and $h(x)$ and outputs the generating function of vector $\mathbf{s}$.

\smallskip

\begin{quote}
\begin{enumerate}
\item[\textbf{Step 1.}] Let $f(t,x)=x\, (R^{n-1}h')((t-1)\, x+1)$.

\item[\textbf{Step 2.}] We have $\sum_{k=1}^n s_k\, x^k = \int_0^1 (R_1^{n-1}f)(t,x)\, dt$, where $R_1^{n-1}$ is the $(n-1)$-reflection with respect to the first argument.
\end{enumerate}
\end{quote}

\smallskip
\end{algorithm}

The computation of $h(x)$ from $\bfs$ or $\bfS$ can be useful if we want to compute the system reliability $h(p)$ directly from the signature and the component reliability $p$.

We already know that Eq.~(\ref{eq:aersd76fds}) gives the polynomial $h(x)$ in terms of vector $\bfS$. The following proposition yields simple expressions of $h(x)$ and $h'(x)$ in terms of vector $\bfs$. This result was already presented in \cite[Sect.~4]{Mar06} and \cite[Rem.~2]{MarMatb} in alternative forms.

\begin{proposition}
We have
\begin{eqnarray}
h'(x) &=& \sum_{k=1}^n s_k{\,}k{\,}{n\choose k}{\,}x^{n-k}(1-x)^{k-1}{\,},\label{eq:45tz6hjk}\\
h(x) &=& \sum_{k=1}^n s_k{\,}I_x(n-k+1,k) ~=~ \sum_{k=1}^n s_k{\,}\sum_{i=n-k+1}^n{n\choose i}{\,}x^i(1-x)^{n-i}{\,},\label{eq:45tz6}
\end{eqnarray}
where $I_x(a,b)$ is the regularized beta function defined, for any $a,b,x>0$, by
$$
I_x(a,b) ~=~ \frac{\int_0^xt^{a-1}(1-t)^{b-1}\, dt}{\int_0^1t^{a-1}(1-t)^{b-1}\, dt}{\,}.
$$
\end{proposition}

\begin{proof}
Formula (\ref{eq:45tz6hjk}) immediately follows from Eq.~(\ref{eq:sfd6543}). Then, from Eqs.~(\ref{eq:Betakn}) and (\ref{eq:45tz6hjk}) we immediately derive the first equality in Eq.~(\ref{eq:45tz6}) since $h(x) = \int_0^x h'(t)\, dt$. The second equality follows from Eqs.~(\ref{eq:SFromS}) and (\ref{eq:aersd76fds}).
\end{proof}

The following proposition provides alternative expressions of $h(x)$ and $h'(x)$ in terms of $\bfS$ and $\bfs$, respectively.

\begin{proposition}\label{prop:DiffTablexX}
We have
\begin{eqnarray}
h(x) &=& \big((x\Delta_i + I)^n{\,}\overline{S}_{n-i}\big)\big|_{i=0}{\,},\label{eq:ds67ds}\\
h'(x) &=& n\big((x\Delta_i + I)^{n-1}{\,}s_{n-i}\big)\big|_{i=0}{\,},\label{eq:ds67ds7}
\end{eqnarray}
where $I$ denote the identity operator.
\end{proposition}

\begin{proof}
By Eq.~(\ref{eq:fuji2p}) we have
$$
h(x) ~=~ \sum_{k=0}^nd_k{\,}x^k ~=~ \sum_{k=0}^n{n\choose k}{\,}x^k{\,}\big(\Delta_i^k{\,}\overline{S}_{n-i}\big)\big|_{i=0}{\,},
$$
which proves Eq.~(\ref{eq:ds67ds}) as we can immediately see by formally expanding the binomial operator expression $(x\Delta_i + I)^n$. Equation~(\ref{eq:ds67ds7}) then immediately follows from Eq.~(\ref{eq:ds67ds}).
\end{proof}

Proposition~\ref{prop:DiffTablexX} shows that the functions $h(x)$ and $h'(x)$ can be computed from difference tables. Setting
$$
D_{j,k}(x)=((x\Delta_i + I)^k{\,}\overline{S}_{n-i})|_{i=j}\quad\mbox{and}\quad d_{j,k}(x)=n((x\Delta_i + I)^{k-1}{\,}s_{n-i})|_{i=j-1},
$$
we can derive the following algorithms for the computation of $h(x)$ and $h'(x)$.

\begin{algorithm}\label{algo:s1x}
The following algorithm inputs vector $\bfS$ and outputs function $h(x)$. It uses the functions $D_{j,k}(x)$ for $k=0,\ldots,n$ and $j=0,\ldots,n-k$.

\smallskip

\begin{quote}
\begin{enumerate}
\item[\textbf{Step 1.}] For $j=0,\ldots,n$, set $D_{j,0}(x):=\overline{S}_{n-j}$.

\item[\textbf{Step 2.}] For $k=1,\ldots,n$

\mbox{}\hspace{3ex}For $j=0,\ldots,n-k$

\mbox{}\hspace{6ex}$D_{j,k}(x) ~:=~ x{\,}D_{j+1,k-1}(x)+(1-x){\,}D_{j,k-1}(x)$

\item[\textbf{Step 3.}] $h(x):=D_{0,n}(x)$.
\end{enumerate}
\end{quote}

\smallskip
\end{algorithm}

\begin{algorithm}\label{algo:s1xg}
The following algorithm inputs vector $\bfs$ and outputs function $h'(x)$. It uses the functions $d_{j,k}(x)$ for $k=1,\ldots,n$ and $j=1,\ldots,n-k+1$.

\smallskip

\begin{quote}
\begin{enumerate}
\item[\textbf{Step 1.}] For $j=1,\ldots,n$, set $d_{j,1}(x):=n{\,}s_{n-j+1}$.

\item[\textbf{Step 2.}] For $k=2,\ldots,n$

\mbox{}\hspace{3ex}For $j=1,\ldots,n-k+1$

\mbox{}\hspace{6ex}$d_{j,k}(x) ~:=~ x{\,}d_{j+1,k-1}(x)+(1-x){\,}d_{j,k-1}(x)$

\item[\textbf{Step 3.}] $h'(x):=d_{1,n}(x)$.
\end{enumerate}
\end{quote}

\smallskip
\end{algorithm}

Table~\ref{tab:345s} summarizes the main conversion formulas obtained thus far. They are given by the corresponding equation numbers. For instance, formulas to compute $\bfs$ from $\bfd$ or $h(x)$ are given in Eqs.~(\ref{eq:fuji3}), (\ref{eq:s87sf6sd}), (\ref{eq:sfrrd6543}), and (\ref{prop:sdf67g}).
\begin{table}[htbp]
\begin{tabular}{|c|ccc|}\cline{2-4}
\multicolumn{1}{c|}{} & $\bfd$ or $h(x)$ & $\bfs$ & $\bfS^{\mathstrut}$ \\
\hline
$\bfd$ or $h(x)$ & & \hspace{1ex}(\ref{eq:fuji5})(\ref{eq:fuji4})(\ref{eq:45tz6})\hspace{1ex} & (\ref{eq:fuji2})(\ref{eq:fuji2p})(\ref{eq:aersd76fds})(\ref{eq:ds67ds})\\
$\bfs$ & (\ref{eq:fuji3})(\ref{eq:s87sf6sd})(\ref{eq:sfrrd6543})(\ref{prop:sdf67g}) & & (\ref{eq:sFromS})(\ref{eq:sfd76dd}) \\
$\bfS$ & (\ref{eq:fuji1})(\ref{eq:lj3kl1})(\ref{eq:sfd76ddd}) & (\ref{eq:SFromS})(\ref{eq:sfd76dd}) & \\
\hline
\end{tabular}
\bigskip
\caption{Conversion formulas}
\label{tab:345s}
\end{table}

\subsection{Conversions based on the dual structure}
\label{sec:dual5}

We end this section by giving conversion formulas involving the dual structure of the system. Let $\phi^D\colon\{0,1\}^n\to\{0,1\}$ be the dual structure function defined as $\phi^D(\bfx) = 1-\phi(\mathbf{1}-\bfx)$, where $\mathbf{1}-\bfx=(1-x_1,\ldots,1-x_n)$, and let $h^D\colon [0,1]^n\to\R$ be its corresponding reliability function, that is, $h^D(\bfx) = 1-h(\mathbf{1}-\bfx)$.

Straightforward computations yield the following conversion formulas, where the upper index $D$ always refers to the dual structure and $\delta$ stands for the Kronecker delta:
\begin{eqnarray}
d_k^D &=& \delta_{k,0}-(-1)^k\sum_{j=k}^n{j\choose k}{\,}d_j{\,},\qquad k=0,\ldots,n,\label{eq:dd1}\\
d_k &=& \delta_{k,0}-(-1)^k\sum_{j=k}^n{j\choose k}{\,}d_j^D{\,},\qquad k=0,\ldots,n,\label{eq:dd2}\\
\overline{S}_k &=& 1-\overline{S}_{n-k}^D ~=~ 1-\sum_{j=0}^k\frac{{k\choose j}}{{n\choose j}}{\,}d_j^D{\,},\qquad k=0,\ldots,n,\label{eq:dd3}\\
s_k &=& s_{n-k+1}^D ~=~ \sum_{j=1}^k\frac{{k-1\choose j-1}}{{n\choose j}}{\,}d_j^D{\,},\qquad k=1,\ldots,n,\label{eq:dd4}\\
d_k^D &=& \delta_{k,0}-{n\choose k}\big(\Delta_i^kS_i\big)\big|_{i=0}{\,},\qquad k=0,\ldots,n,\label{eq:dd5}\\
d_k^D &=& {n\choose k}\,\big(\Delta_i^{k-1}{\,}s_{i}\big)\big|_{i=1}{\,},\qquad k=1,\ldots,n.\label{eq:dd6}
\end{eqnarray}

Recall that $\phi_k$ gives the number of path sets of size $k$. Combining (\ref{eq:TailSignature2}) with (\ref{eq:asd76fds}), we obtain the identity $\sum_{k=0}^n\phi_{n-k}{\,}x^k = (R^nh)(x+1)$, from which we immediately derive the following generating function
$$
\sum_{k=0}^n\phi_k{\,}x^k ~=~ R^n((R^nh)(x+1)) ~=~ (x+1)^n{\,}h\Big(\frac{x}{x+1}\Big){\,}.
$$
Note that this function can also be obtained by using Eqs.~(\ref{eq:TailSignature2}), (\ref{eq:dd3}), and the dual version of Eq.~(\ref{eq:asd76fds}). Indeed, we have
\begin{eqnarray*}
\sum_{k=0}^n\phi_k\, x^k &=& \sum_{k=0}^n {n\choose k}\overline{S}_{n-k}{\,}x^k ~=~ \sum_{k=0}^n {n\choose k}{\,}x^k-\sum_{k=0}^n {n\choose k}\overline{S}^D_{k}{\,}x^k\\
&=& (x+1)^n-(R^nh^D)(x+1).
\end{eqnarray*}

\section{The general dependent case}

In this final section we drop the i.i.d.\ assumption and consider the general dependent setting, assuming only that there are no ties among the component lifetimes (i.e., $\Pr(T_i=T_j)=0$ whenever $i\neq j$). As a consequence, the function $h(\bfx)$ may no longer express the reliability of the system in terms of the component reliabilities.

Two concepts of signature emerge in this general setting. First, we can consider the {\em structure signature}, that is, the $n$-vector $\bfs =(s_1,\ldots,s_n)$ whose $k$-th coordinate is given by Boland's formula (\ref{eq:asad678}). Of course, the conversion formulas and algorithms obtained in Sections 2 and 3 can still be used ``as is'', even if the i.i.d.\ assumption is dropped. Second, we can consider the {\em probability signature}, that is, the $n$-vector $\mathbf{p}=(p_1,\ldots,p_n)$ whose $k$-th coordinate is given by $p_k=\Pr(T_S=T_{k:n})$.

We now elaborate on this latter case and show that a modification of the structure function enables us to formally extend almost all the conversion formulas and algorithms obtained in Sections 2 and 3 to the general dependent setting.

It was recently shown \cite{MarMat11} that
\begin{equation}\label{eq:asad678BolQ}
p_k ~=~ \sum_{\textstyle{A\subseteq C\atop |A|=n-k+1}}q(A)\,\phi(A)-\sum_{\textstyle{A\subseteq C\atop |A|=n-k}}q(A)\,\phi(A)\, ,
\end{equation}
where the function $q\colon 2^{[n]}\to\R$, called the {\em relative quality function} associated with the system, is defined by
$$
q(A) ~=~ \Pr\left(\max_{i\notin A}T_i<\min_{i\in A}T_i\right),
$$
and has the property $\sum_{|A|=k}q(A)=1$ for $k=0,\ldots,n$. Thus, for any subset $A\subseteq C$, the number $q(A)$ is the probability that the best $|A|$ components of the system are precisely those in $A$.

In the special case when the component lifetimes are i.i.d., or even exchangeable, the number $q(A)$ is exactly $1/{n\choose |A|}$ and therefore by comparing Eqs.~(\ref{eq:asad678}) and (\ref{eq:asad678BolQ}) we immediately see that the vector $\mathbf{p}$ then reduces to $\bfs$. As mentioned in \cite{MarMat11}, this observation motivates the introduction of the {\em normalized relative quality function} $\tilde{q}\colon 2^{[n]}\to\R$, defined by $\tilde{q}(A) = {n\choose |A|}{\,}q(A)$. We then have $\tilde{q}(A)=1$ whenever the component lifetimes are i.i.d.\ or exchangeable.

Following a suggestion by P. Mathonet, we now assign to the system a pseudo-structure function $\psi\colon 2^{[n]}\to\R$ defined so as to have
\begin{equation}\label{eq:fd994jncj}
\sum_{\textstyle{A\subseteq C\atop |A|=k}}\frac{1}{{n\choose |A|}}\,\psi(A) ~=~\sum_{\textstyle{A\subseteq C\atop |A|=k}}q(A)\,\phi(A),\qquad k=0,\ldots,n.
\end{equation}

\begin{definition}\label{de:76dfds5df}
Let $(C,\phi)$ be an $n$-component system with relative quality function $q$. The {\em $q$-structure function} associated with the system is the set function $\psi\colon 2^{[n]}\to\R$ defined by
$$
\psi(A) ~=~ \tilde{q}(A){\,}\phi(A) ~=~
\begin{cases}
{n\choose |A|}{\,}q(A), & \mbox{if $A$ is a path set},\\
0, & \mbox{otherwise}.
\end{cases}
$$
\end{definition}

It is clear that $\psi$ reduces to $\phi$ whenever the component lifetimes of the system are i.i.d.\ or exchangeable. In the general dependent case, the function $\psi$ is a pseudo-Boolean function, that is, a function from $\{0,1\}^n$ to $\R$. As such, it has the following multilinear form
$$
\psi(\bfx) ~=~ \sum_{A\subseteq C} \psi(A){\,}\prod_{i\in A}x_i{\,}\prod_{i\in C\setminus A}(1-x_i){\,},\qquad\bfx\in\{0,1\}^n.
$$
Just as $h(\bfx)$ is the multilinear extension of $\phi(\bfx)$, we can also define the multilinear extension $g\colon [0,1]^n \to\R$ of $\psi(\bfx)$; that is,
$$
g(\bfx) ~=~ \sum_{A\subseteq C} \psi(A){\,}\prod_{i\in A}x_i{\,}\prod_{i\in C\setminus A}(1-x_i){\,},\qquad\bfx\in [0,1]^n.
$$
This function can always be put in the unique standard multilinear form
\begin{equation}\label{eq:345hg27w}
g(\bfx) ~=~ \sum_{A\subseteq C}c(A)\,\prod_{i\in A}x_i{\,},
\end{equation}
where, by the M\"obius inversion theorem, the coefficient $c(A)$ is given by
\begin{equation}\label{eq:mob345hg27w}
c(A) ~=~ \sum_{B\subseteq A} (-1)^{|A|-|B|}{\,}\psi(B).
\end{equation}

Thus, in this general setting we readily see that Eq.~(\ref{eq:fd994jncj}) holds and, consequently, that Eq.~(\ref{eq:asad678}) immediately extends to Eq.~(\ref{eq:asad678BolQ}).

Now, if we also define the values $\overline{P}_k$, $\psi_k$, and $c_k$ for $k=0,\ldots,n$ as
$$
\overline{P}_k ~=~ \sum_{i=k+1}^n p_i{\,},\qquad \psi_k ~=~ \sum_{\textstyle{A\subseteq C\atop |A|=k}}\psi(A){\,},\qquad c_k ~=~ \sum_{\textstyle{A\subseteq C\atop |A|=k}}c(A){\,},
$$
then we can formally extend all our formulas and algorithms from Eq.~(\ref{eq:SFromS}) to Eq.~(\ref{eq:ds67ds7}) \emph{mutatis mutandis} to the general dependent setting.

More formally, we have the following straightforward theorem.

\begin{theorem}\label{thm:main5DC}
Equations (\ref{eq:SFromS})--(\ref{eq:ds67ds7}) still hold if we replace $s_k$, $\overline{S}_k$, $h(\bfx)$, $h(x)$, $\phi(A)$, $d(A)$, $\phi_k$, and $d_k$ with $p_k$, $\overline{P}_k$, $g(\bfx)$, $g(x)$, $\psi(A)$, $c(A)$, $\psi_k$, and $c_k$, respectively.
\end{theorem}

Let us illustrate how this theorem can be applied. Considering for instance Eq.~(\ref{eq:TailSignature2}), Theorem~\ref{thm:main5DC} shows that this equation can be translated in the general setting into
$$
\overline{P}_k ~=~ \frac{1}{{n\choose k}}\,\sum_{\textstyle{A\subseteq C\atop |A|=n-k}}\,\psi(A) ~=~ \frac{1}{{n\choose k}}\,\psi_{n-k}{\,},\qquad k=0,\ldots,n.
$$
Similarly, from Eq.~(\ref{eq:sfrrd6543}) we immediately derive the identity
\begin{equation}\label{eq:exsg3fghf}
\sum_{k=1}^n{n\choose k}{\,}p_k{\,}x^k ~=~ \int_0^x(R^{n-1}g')(t+1)\, dt{\,}.
\end{equation}

\begin{example}\label{ex:final}
Consider a $3$-component system whose structure function is given by
$$
\phi(x_1,x_2,x_3) ~=~ x_1(x_2\amalg x_3) ~=~ x_1x_2x_3 + x_1x_2(1-x_3) + x_1(1-x_2){\,}x_3{\,}.
$$
The $q$-structure function is then given by
$$
\psi(x_1,x_2,x_3) ~=~ x_1x_2x_3+3{\,}q(\{1,2\}){\,}x_1x_2(1-x_3)+3{\,}q(\{1,3\}){\,}x_1(1-x_2){\,}x_3{\,}.
$$
Using Eq.~(\ref{eq:exsg3fghf}), we finally obtain
\begin{eqnarray*}
p_1 &=& 1-q(\{1,2\})-q(\{1,3\}) ~=~ q(\{2,3\}){\,},\\
p_2 &=& q(\{1,2\})+q(\{1,3\}){\,},\\
p_3 &=& 0{\,}.
\end{eqnarray*}
\end{example}

It is noteworthy that in practice the function $g(\bfx)$ is much heavier to handle than the function $h(\bfx)$ (consider for instance Example~\ref{ex:21j4hg}). Moreover, the function $g(\bfx)$ need not be nondecreasing in each argument and hence it cannot be easily expressed as a coproduct over the minimal path sets.

However, despite these observations, Theorem~\ref{thm:main5DC} shows that this formal extension of the conversion formulas is mathematically elegant and might have theoretical applications.

\section{Concluding remarks}

We have provided various conversion formulas between the signature and the reliability function for systems with continuous and i.i.d.\ component lifetimes and we have extended theses formulas to the general dependent case. This study can be regarded as the continuation of paper \cite{MarMatb}, where Eqs.~(\ref{eq:asd76fds})--(\ref{eq:lj3kl1}), Algorithm~\ref{algo:s55}, and Proposition~\ref{prop:sad6f} were already presented and established.

We conclude this paper with the following two observations which are worth particular mention:
\begin{itemize}
\item It is a well-known fact that, under the i.i.d.\ assumption, both the structure signature $\bfs =(s_1,\ldots,s_n)$ and the reliability function $h(\bfx)$ are purely combinatorial objects associated with the structure function of the system. As a consequence, the developments and results presented in Sections 2 and 3 are based only on combinatorial and algebraic arguments and do not really require any stochastic setting, even if such a setting has to be considered to define the component lifetimes.

\item The $q$-structure function of a system as introduced in Definition~\ref{de:76dfds5df} is simply a convenient transformation of the structure function of the system which enables us to extend Equations (\ref{eq:SFromS})--(\ref{eq:ds67ds7}) to the general dependent case. Even though the $q$-structure function $\psi(\bfx)$ and its corresponding multilinear extension $g(\bfx)$ are heavier to handle than their i.i.d.\ counterparts $\phi(\bfx)$ and $h(\bfx)$, Theorem~\ref{thm:main5DC} suggests that this extension is interesting more from a conceptual than applied viewpoint.
\end{itemize}

\section*{Acknowledgments}

The author would like to thank the reviewer for timely and helpful suggestions for improving the organization of this paper. This research is supported by the internal research project F1R-MTH-PUL-12RDO2 of the University of Luxembourg.


\begin{thebibliography}{1}


\bibitem{BarIye88}
R.~E.~Barlow and S.~Iyer.
\newblock Computational complexity of coherent systems and the reliability polynomial.
\newblock {\em Probability in the Engineering and Informational Sciences}, 2:461--469, 1988.

%
\bibitem{BarPro81}
R.~E.~Barlow and F.~Proschan.
\newblock {\em Statistical theory of reliability and life testing}.
\newblock To Begin With, Silver Spring, MD, 1981.

\bibitem{Bol01}
P.~J. Boland.
\newblock Signatures of indirect majority systems.
\newblock {\em J. Appl. Prob.}, 38:597--603, 2001.

\bibitem{BolSamVes03}
P.~J. Boland, F.~J. Samaniego, and E.~M. Vestrup.
\newblock {\em Linking dominations and signatures in network reliability theory}.
\newblock In: Mathematical and statistical methods in reliability (Trondheim, 2002), pages 89-–103.
\newblock Ser. Qual. Reliab. Eng. Stat., 7, World Sci. Publ., River Edge, NJ, 2003.

\bibitem{GerShpSpi11}
I.~Gertsbakh, Y.~Shpungin, and F.~Spizzichino.
\newblock Signatures of coherent systems built with separate modules.
\newblock {\em J. Appl. Probab.}, 48(3):843--855, 2011.


%

\bibitem{Mar06}
J.-L.~Marichal.
\newblock Cumulative distribution functions and moments of lattice polynomials.
\newblock {\em Stat. Prob. Letters}, 76(12):1273--1279, 2006.

\bibitem{MarMat11}
J.-L.~Marichal and P.~Mathonet.
\newblock Extensions of system signatures to dependent lifetimes: Explicit expressions and interpretations.
\newblock {\em J. Multivariate Analysis}, 102(5):931--936, 2011.

%

%
\bibitem{MarMatb}
J.-L.~Marichal and P.~Mathonet.
\newblock Computing system signatures through reliability functions.
\newblock {\em Stat. Prob. Letters}, 83(3):710--717, 2013.


%


\bibitem{Ram90}
K.~G. Ramamurthy.
\newblock {\em Coherent structures and simple games}, volume~6 of {\em Theory
  and Decision Library. Series C: Game Theory, Mathematical Programming and
  Operations Research}.
\newblock Kluwer Academic Publishers Group, Dordrecht, 1990.

\bibitem{Sam85}
F.~J.~Samaniego.
\newblock On closure of the IFR class under formation of coherent systems.
\newblock {\em IEEE Trans.\ Reliability}, 34:69--72, 1985.

\bibitem{Sam07}
F.~J.~Samaniego.
\newblock {\em {System signatures and their applications in engineering reliability}.}
\newblock {Int. Series in Operations Research \& Management Science, 110. New York: Springer}, 2007.

%


\end{thebibliography}
\end{document}